%% file: p-adicIncidence.tex
\documentclass[12pt]{amsart}
\input{preamble}

\title{Point-Line Incidence Estimates in $(\ZZ/p^k\ZZ)^2$}
\author{Yuhan Chu}
\address{University of British Columbia, Vancouver, BC, Canada}
\email{ychu12@student.ubc.ca}
\date{\today}

\begin{document}
\maketitle

\begin{abstract}
    The point-line incidence problem has been widely studied in Euclidean spaces and vector spaces over finite fields, whereas the analogous problem has rarely been considered over finite $p$-adic rings. In this paper, we investigate incidences in the $p$-adic setting and prove new incidence bounds for points and lines in $(\ZZ/p^k\ZZ)^2$. Our first two results extend previously known incidence bounds over finite fields, assuming lines are well-separated. For non-separated lines, we establish a general incidence result for weighted points and lines under certain dimensional spacing conditions using the Fourier analytic method and the induction-on-scales argument.
\end{abstract}

\smallskip

\section{Introduction}
Let $\CP$ be a finite set of points and $\CL$ a finite set of lines. We define the set of incidences between $\CP$ and $\CL$ by 
\begin{equation*}
    \CI(\CP,\CL) = \{(q,l)\in \CP\times \CL: q\in l\}.
\end{equation*}
This leads to a natural question of the maximal number of incidences. For points and lines in the Euclidean plane $\RR^2$, Szemer\'{e}di and Trotter \cite{SzTr} proved the classical result:
\begin{equation}\label{SzTrE}
    |\CI(\CP,\CL)| \lesssim |\CP|^{\frac{2}{3}}|\CL|^{\frac{2}{3}} + |\CP| + |\CL|,
\end{equation}
where the notation $X\lesssim Y$ means $X \leq CY$ for some constant $C > 0$. In particular, (\ref{SzTrE}) is sharp up to a constant due to a construction by Elekes \cite{Elekes}.\\
The Szemer\'{e}di-Trotter theorem is considered a fundamental result in discrete geometry and extremal combinatorics. In addition to Szemer\'{e}di and Trotter's original proof, the bound (\ref{SzTrE}) can be proved in various methods. Most famously, Sz\'{e}leky \cite{Szekely} pointed out a beautiful argument to prove (\ref{SzTrE}) via the crossing number inequality. Cell decomposition is another classical method for addressing incidence problems, first introduced by Clarkson, Edelsbrunner, Guibas, Sharir, and Welzl in \cite{CEGSW90}. More recently, this method can be achieved in a more practical way to prove (\ref{SzTrE}) through Guth and Katz's polynomial partitioning argument \cite{GK15}. All these methods rely on some special geometric or topological properties in $\RR^2$, and therefore are difficult to apply to other settings.\\
Over finite prime fields $\FF_p$, a direct example consisting of all points and all non-vertical lines in $\FF_p^2$ shows that the Szemer\'{e}di-Trotter theorem does not hold generally in finite fields. Consequently, additional restrictions are required for nontrivial results. Bourgain, Katz, and Tao \cite{BKT04} were the first to establish a nontrivial incidence bound in $\FF_p^2$. They proved that if $|\CP|,|\CL| \leq N = p^{\alpha}$ for some $\alpha \in (0,2)$, then 
\begin{equation*}
    |\CI(\CP,\CL)| \lesssim N^{\frac{3}{2}-\varepsilon}
\end{equation*}
for some $\varepsilon = \varepsilon(\alpha)$ only depending on $\alpha$. Later, Helfgott and Rudnev \cite{HR11} were able to give an explicit value $\varepsilon = 1/10678$ for Bourgain-Katz-Tao's result, and the value of $\varepsilon$ was further improved by Jones in \cite{Jones16} and \cite{Jones12} to $\varepsilon = 1/662$.\\
In the finite field setting, some extreme cases have been relatively well-studied. For large sets of points and lines, say $|\CP|,|\CL| > p$, Vinh \cite{Vinh11} obtained the following bound via tools from spectral graph theory:
\begin{equation*}
    \left|\CI(\CP,\CL) - \frac{|\CP||\CL|}{p}\right| \leq p^{\frac{1}{2}}|\CP|^{\frac{1}{2}}|\CL|^{\frac{1}{2}}.
\end{equation*}
For small sets of points, Grosu \cite{Grosu14} proved that for $|\CP| \sim \log\log\log p$, we could attain an incidence bound $|\CI(\CP,\CL)| \lesssim |\CP|^{\frac{4}{3}}$ close to the classical Szemer\'{e}di-Trotter theorem. For non-extreme cases, Stevens and de Zeeuw proved that if $|\CP|^{\frac{7}{8}} \lesssim |\CL| \lesssim |\CP|^{\frac{8}{7}}$ and $|\CP|^{-2}|\CL|^{13} \lesssim p^{15}$, we have the incidence bound 
\begin{equation*}
    |\CI(\CP,\CL)| \lesssim |\CP|^{\frac{11}{15}}|\CL|^{\frac{11}{15}}.
\end{equation*}
This result, when $|\CP| = |\CL| = N$, is better than the result in \cite{Vinh11} for $N \lesssim p^{\frac{15}{14}}$, and improves $\varepsilon$ in \cite{Jones12} from $1/662$ to $1/30$. More general non-extreme cases have not yet been fully understood in the literature.\\
The main focus of this paper is to study the point-line incidence problem in the $p$-adic setting, the incidences between points and lines in modules over finite rings $R_k \coloneqq \ZZ/p^k\ZZ$. Finite $p$-adic rings serve as an intermediate state between finite fields and Euclidean spaces. The rings $R_k$ are still discrete, as finite fields, but they also admit multiple scales similar to the Euclidean setting. However, $p$-adic analogues of incidence problems are rarely studied in the literature. The existing results in the setting of the finite cyclic ring $\ZZ/m\ZZ$, all proved using Erd\H{o}s-R\'{e}nyi graphs in the same spirit as the graph-theoretical method in \cite{Vinh11}, tend to be mostly effective when $m$ has only a few prime divisors, and when $\CP$ contain almost all points and $\CL$ contains almost all lines in $(\ZZ/m\ZZ)^2$. In \cite{TV15}, Thang and Vinh obtained the following Szemer\'{e}di-Trotter type theorem:
\begin{equation}\label{TVbd}
    \left|\CI(\CP,\CL) - \frac{|\CP||\CL|}{m} \right| \leq \frac{2\tau(m)m^2}{\phi(m)\gamma(m)^{\frac{1}{2}}}\sqrt{|\CP||\CL|},
\end{equation}
where $\gamma(m)$ is the smallest prime divisor of $m$, $\tau(m)$ is the number of divisors of $m$, and $\phi(m)$ is Euler's totient function. When $m = p^r$ and $|\CP|,|\CL| \lesssim N = p^{\alpha}$ with $2r-1+\varepsilon \leq\alpha \leq 2r-\varepsilon$, (\ref{TVbd}) becomes a reasonably good estimate 
\begin{equation*}
    |\CI(\CP,\CL)| \lesssim N^{\frac{3}{2}-\frac{\varepsilon}{4r}}.
\end{equation*}
Vinh also provided in \cite{Vinh14} a probabilistic characterization of the existence of incidences in $(\ZZ/m\ZZ)^2$ alongside their major result for finite fields. Moreover, Pham and Vinh gave in \cite{PV17} an incidence estimate essentially equivalent to (\ref{TVbd}) for $m = p^r$ as an intermediate step towards their estimates on the number of triangle areas.\\
In this paper, we investigate the incidences for points and lines in $(\ZZ/p^k\ZZ)^2$ under several different conditions that have not been considered so far in this setting. Our first two results are for well-separated lines in $(\ZZ/p^k\ZZ)^2$, which are extensions of results in \cite{BKT04} and \cite{SZ17} for finite fields. The definition of the separation condition on lines will be given later in this section.

\begin{theorem}\label{BKTanalogueThm}
    Let $\CP$ be a set of points of $R_k^2$ and $\CL$ a set of lines in $R_k^2$ that are either $1$-separated in direction or $1$-separated in distance. Suppose $|\CP|,|\CL| \leq N = p^{\alpha}$ for some $0 < \alpha < 2$. Then 
    \begin{equation*}
        |\CI(\CP,\CL)| \lesssim N^{\frac{3}{2}-\varepsilon},
    \end{equation*}
    where $\varepsilon = \varepsilon(\alpha) > 0$ depends only on $\alpha$.
\end{theorem}

\begin{theorem}\label{SZp-adicThm}
    Let $\CP$ be a set of $m$ points in $R_k^2$ and $\CL$ a set of $n$ lines in $R_k^2$ such that the lines in $\CL$ are either $1$-separated in direction or $1$-separated in distance. Suppose that $n^{\frac{7}{8}} < m < n^{\frac{8}{7}}$ and $m^{-2}n^{13} \lesssim p^{15}$. Then 
    \begin{equation*}
        |\CI(\CP,\CL)| \lesssim |\CP|^{\frac{11}{15}}|\CL|^{\frac{11}{15}}.
    \end{equation*}
\end{theorem}


The proofs of Theorem \ref{BKTanalogueThm} and \ref{SZp-adicThm} closely follow the arguments in the aforementioned papers with some modifications. The typical arguments for Szemer\'{e}di-Trotter type results in finite fields proceed as follows. We first refine the sets of points and lines through iterative applications of the pigeonhole principle and find two popular points such that the intersections of lines through these points capture the majority of incidences, which we call the grid-like structure. Then we conduct a projective transformation which transforms the grid-like structure into a Cartesian product configuration by mapping the two selected points to points at infinity, and exploit the additive and multiplicative information to approach a contradiction. These arguments depend heavily on the validity of projective transformations that map arbitrary two distinct points to points at infinity, which unfortunately does not hold generally in the $p$-adic setting due to the existence of nontrivial zero divisors. Although projective transformations in $R_k^2$ fail to work effectively, after rescaling and projection, every such grid-like structure can be reduced to configurations in $(\ZZ/p\ZZ)^2$ without losing the information of incidences captured by the structure. This modification allows us to continue with the usual projective transformations in finite fields and arrive at the same conclusion.\\ 
We are also interested in incidences for non-separated lines that may intersect at more than one point. This appears with more generality in the $p$-adic setting, in which case the multiscale nature of $p$-adic rings aligns better with the multiple scales in the Euclidean setting. In particular, we are able to follow the induction-on-scales argument frequently applied to study incidences for tubes and the Furstenberg set problem in the Euclidean plane.\\
However, unlike the Euclidean spaces, there are only finitely many scales in spaces over $p$-adic rings, leading to the fact that distinct cubes or tubes may become identical after thickening to lower scales. To overcome this obstruction, we introduce sets of weighted points and lines and consider the generalized incidence, which is the sum of the products of the weights of incident points and lines, by virtue of a general incidence result in \cite{Bradshaw23} for weighted atoms and tubes in the Euclidean spaces. Similar weighted incidences are also investigated later in \cite{CPZ23} for weighted points and lines, and in \cite{RW23} for weighted dyadic cubes and tubes.\\
We provide a generalized incidence estimate for weighted points and lines in $R_k^2$ obeying certain dimensional spacing conditions. A set of weighted points or lines is called a $(p^{-j},\alpha,K_{\alpha})$-set if the points or lines are at scale $p^{-j}$ (so they are identified with $p^{-j}$-cubes or tubes in $R_k^2$), and the set resembles an $\alpha$-dimensional subset with respect to any lower scale $p^{-\ell}$, $\ell \leq j$, with $K_{\alpha}$ being some constant indicating the maximal size. We defer the rigorous definitions to Section 3. The main result on the generalized incidence under the dimensional spacing conditions is as follows.
\begin{theorem}\label{weighted}
    Fix $\varepsilon \in (0,1)$ and $0\leq \alpha,\beta \leq 2$. Let $c^{-1} = \max(\alpha+\beta-1,2)$. Then there exists a constant $C(p,\varepsilon) > 0$ such that for any $(p^{-k},\alpha,K_{\alpha})$-set of points $\CP$ and $(p^{-k},\beta,K_{\beta})$-set of lines $\CL$,
    \begin{equation*}
        I_w(\CP,\CL) \leq C(p,\varepsilon)p^{k(c+\varepsilon)}(K_{\alpha}K_{\beta})^c|\CP|_w^{1-c}|\CL|_w^{1-c},
    \end{equation*}
    where the constant $C(p,\varepsilon)$ is chosen so that 
    \begin{equation*}
        C(p,\varepsilon) > \max\{p^{1-\varepsilon}, (p^{\varepsilon}- 1)^{-1}\}.
    \end{equation*}
    In particular, if $\CP$ and $\CL$ are sets of unweighted points and lines, we have the incidence estimate:
    \begin{equation*}
        I(\CP,\CL) \leq C(p,\varepsilon)p^{k(c+\varepsilon)}(K_{\alpha}K_{\beta})^c|\CP|^{1-c}|\CL|^{1-c}.
    \end{equation*}
\end{theorem}
Theorem \ref{weighted} is proved via the $p$-adic analogue of the Fourier analytic high-low method and the induction-on-scales argument. We will constantly apply the identification of the $p^{-j}$-cubes and tubes in $R_k^2$ with points and lines in $R_j^2$ for $j\leq k$, so that thickening points and lines in $R_k^2$ to $p^{-j}$-cubes and tubes is equivalent to moving from points and lines in $R_k^2$ to those in $R_j^2$. The key principle in the proof is that for sets of weighted points and lines, the total sum of weights is invariant under thickening. This provides more consistency in the behavior of sets of points and lines when moving from higher to lower scales. Additionally, a new incidence estimate will be obtained as a special case.

\subsection{Notations}
Throughout the paper, we use $R_k$ to denote the finite $p$-adic ring $\ZZ/p^k\ZZ$. For $x = (x_1,x_2),y = (y_1,y_2)\in R_k^2$, we define the $p$-adic distance $\|x-y\|$ as follows. We say that $\|x-y\| = p^{-j}$ for some $0\leq j \leq k$ if $p^j \mid (x-y)$ and $p^{j+1}\nmid (x-y)$; that is, $p^j \mid x_i-y_i$ for $i = 1,2$ and $p^{j+1}\nmid x_i-y_i$ for at least one $i$.\\
Let $\PP R_k^1$ to denote the projective space over $R_k^2$, which is defined as 
\begin{equation*}
    \PP R_k^1 = R_k^{2,\times}/R_1^{\times},
\end{equation*}
where $R_k^{2,\times}$ is the set of elements in $R_k^2$ with at least one invertible component, and $R_1^{\times} = (\ZZ/p\ZZ) \setminus \{0\}$. A nondegenerate line $\ell_b(a)$ in $R_k^2$ for some $a\in R_k^2$ and $b\in \PP R_k^1$ is the set $\{a+tb: t\in R_k\}$. Here we only consider nondegenerate lines that are copies of $R_k$, and this is equivalent to the lines $\ell$ with direction $b\in \PP R_k^1$. For two directions $b,b'\in \PP R_k^1$, we say that the $p$-adic angle of $b$ and $b'$ is at most $p^{-j}$, written as $\angle (b,b') \leq p^{-j}$, if there exist representatives $rb$ and $r'b'$ for $b$ and $b'$, respectively, so that $p^j \mid rb - r'b'$, and we say that $\angle(b,b') = p^{-j}$ if $\angle(b,b') \leq p^{-j}$ but $\angle(b.b') \not\leq p^{-(j+1)}$. If two lines $l$ and $l'$ in $R_k^2$ are in directions $b$ and $b'$, respectively, the angle $\angle(l,l')$ of the lines is defined as the angle of their directions.\\
For $0\leq j \leq k$, a $p^{-j}$-cube $Q_j$ in $R_k^2$ is a set of the form $Q_j = \{y\in R_k^2: p^j\mid (y-x)\}$ for some $x\in R_k^2$, which is the set of points in $R_k^2$ with $p$-adic distance from $x$ less than or equal to $p^{-j}$.\\
Define $\pi_j: R_k^2 \rightarrow R_j^2$, the projection from $R_k^2$ to $R_j^2$ for $0\leq j \leq k$, by 
\begin{equation*}
    \pi_j(x) = x\mod p^j.
\end{equation*}
A $p^{-j}$-tube $T_j$ is a set defined as $\pi_j^{-1}(\ell_j)$ for some line $\ell_j$ in $R_j^2$. Namely, it is the $p^{-j}$-neighborhood of some line $\ell$ in $R_k^2$ ($\ell$ is generally not unique).\\
Let $0\leq j\leq k-1$ and $1\leq r \leq k$. We say that two lines $l$ and $l'$ in $R_k^2$ are $p^{-j}$-separated in direction if $\angle(l,l') \geq p^{-j}$. Two \textit{disjoint} lines $l$ and $l'$ with angle $\angle(l,l') = p^{-r}$ in $R_k^2$ are in distance $p^{-j}$ for some $j < r$ if the distance of the projections $\pi_r(l)$ and $\pi_r(l')$ as parallel lines in $R_r^2$ is $p^{-j}$ (Notice that two lines with $\angle(l,l') = 1$ intersect necessarily, see Lemma 8.4 (i) in \cite{LabaTrainor}). And we call two lines $l$ and $l'$ $p^{-j}$-separated in distance if they are disjoint with distance $\geq p^{-j}$.\\
Notice that for every $x\in R_k^2$ and for $0\leq j \leq k$, $x$ can be uniquely represented as 
\begin{equation*}
    x = x_{*} + p^j x^{*} \mod p^k,\quad \text{with $x_{*}\in [p^j]^2$ and $x_{*}\in [p^{k-j}]^2$}.
\end{equation*}
Here we use $[n] = \{0,1,\dots, n-1\}$ for $n\in \NN$. 
For any $p^{-j}$-cube $Q$ in $R_k^2$ with $0\leq j \leq k$, $Q$ can be identified with $R_{k-j}^2$ via the rescaling map $\iota_{Q}: Q \rightarrow R_{k-j}^2$ given by $\iota_{Q}(x) = x^{*}$.\\
It can be shown that $\pi_j$ are linear and map lines in $R_k^2$ to lines in $R_j^2$. While $\iota_{Q}$ are not linear, they map the intersection of lines in $R_k^2$ with $Q$ to lines in $R_{k-j}^2$. We refer to Section 8 in \cite{LabaTrainor} for a full introduction to the $p$-adic geometric properties used in this paper.

\subsection{Organization}
In Section 2, we study incidences for well-separated lines in $(\ZZ/p^k\ZZ)^2$ and prove Theorem \ref{BKTanalogueThm} and \ref{SZp-adicThm}, highlighting the necessary modifications specific to the $p$-adic setting. In Section 3, we consider non-separated lines, in which case the Euclidean axiom does not hold. We begin with a combinatorial discussion on dimensional spacing conditions. Then we generalize these conditions to sets of weighted points and lines and prove Theorem \ref{weighted} via the Fourier analytic method and induction on scales.

\smallskip

\section{Incidence estimates for well-separated lines}
Let $\CP$ be a set of points of $R_k^2$, and $\CL$ a set of lines in $R_k^2$. Our current focus in this section will be on well-separated lines; that is, we always assume that every two lines $\ell$ and $\ell'$ of $\CL$ are either $1$-separated in direction or $1$-separated in distance. Under this separation condition, we are in a favorable geometry where two lines intersect at most one point. Since we can always embed configurations in $\FF_p^2$ into $R_k^2$, we may not expect better incidence bounds than those in finite fields. Our main results in this section indicate that we can actually maintain equally nice incidence estimates in the $p$-adic setting for this special case.\\
Recall that the modification of the arguments in finite fields reduces the grid-like structure attained from the refinement of sets of points and lines in $R_k^2$ to configurations in $(\ZZ/p\ZZ)^2$. This reduction to finite fields relies on the following lemma, which is the key ingredient in the proofs of Theorem \ref{BKTanalogueThm} and \ref{SZp-adicThm}. This is also a generalization of Lemma 8.6 in \cite{LabaTrainor} to higher scales.
\begin{lemma}\label{sepintersec}
    Let $q,q'\in R_k^2$ with $\|q-q'\| = p^{-j}$ for some $1 \leq j \leq k-1$. Suppose $l$ and $l'$ are lines in $R_k^2$ passing through $q,q'$, respectively, with $\angle (l,l') = 1$. Let $Q_j$ be the $p^{-j}$-cube containing both $q$ and $q'$. Then $l\cap l' \subseteq Q_j$.
\end{lemma}
\begin{proof}[Proof of Lemma \ref{sepintersec}]
    Let $q = (x,y)$ and $q' = (x',y')$. Under the assumption, there exist $c,d$ such that $x' = x + p^j c$, $y' = y + p^j d$ and either $p\nmid c$ or $p\nmid d$. We consider the following cases:\\
    \textit{Case 1}: $l: (x,y) + t(1,b)$ and $l': (x',y') + s(1,b')$ with $p\nmid b-b'$ and $s,t\in [p^k]$. Then the intersection point should satisfy 
    \begin{align*}
        x + t = x' + s = x + p^j c + s \mod p^k,\\
        y + tb = y' + sb' = y + p^j d + sb' \mod p^k.
    \end{align*}
    which implies that 
    \begin{align*}
        t = p^j c + s \mod p^k,\quad \text{and} \quad tb = p^j d + sb' = p^jc b + sb \mod p^k
    \end{align*}
    So we obtain $s(b-b') = p^j(d-cb) \mod p^k$.  Since $p\nmid (b-b')$, $(b-b')$ is invertible in $R_k$, and so 
    $s = p^j(d-cb)(b-b')^{-1} \mod p^k$. It follows that $p^j \mid s$, and the intersection point $q''$ satisfies $\|q''-q'\| \leq p^{-j}$, which means $q'' \in Q_j$.\\
    \textit{Case 2}: w.l.o.g. we may assume that $l: (x,y) + t(1,b)$ and $l': (x',y') + s(rp,1)$ with $s,t\in [p^k]$ for some $b\in [p^k]$ and $r\in [p^{k-1}]$. Then the intersection point should satisfy 
    \begin{align*}
        x + t = x' + rps = x + p^jc + rps \mod p^k,\\
        y + tb = y' + s = y + p^j d + s \mod p^k,
    \end{align*}
    which implies that 
    \begin{equation*}
        t = p^jc + rps \mod p^k \quad \text{and}\quad s = tb - p^j d \mod p^k.
    \end{equation*}
    So we obtain $(rbp-1)s = p^j(d-cb) \mod p^k$. Since $p\nmid (rbp-1)$, $(rbp-1)$ is invertible in $R_k$, and we have 
    $s = p^j(d-cb)(rbp-1)^{-1} \mod p^k$. Hence $p^j \mid s$ and again we have the intersection point $q''$ satisfies that $\|q''-q'\| \leq p^{-j}$, which means $q''\in Q_j$.
\end{proof}

\subsection{Proof of Theorem 1.1}
In this section, we prove Theorem \ref{BKTanalogueThm}, which concerns the incidence estimate of sets of points and lines of relatively small size with lines satisfying the separation condition. This is an extension of Theorem 6.2 in \cite{BKT04} for finite fields to $(\ZZ/p^k\ZZ)^2$.
\begin{proof}[Proof of Theorem \ref{BKTanalogueThm}]
    Notice that under the separation condition, the lines of $\CL$ now satisfy that two lines intersect in at most one point, and as a result, there is a unique line of $\CL$ (if it exists) passing through two distinct points of $\CP$. Thus, we are able to follow the popularity argument in \cite{BKT04} to refine the sets of points and lines and extract popular points for the grid-like structure. By adding dummy points and lines, we may assume that $|\CP| = |\CL| = N$. We may assume that $N$ is large (otherwise we can always pick a large enough constant $C$ so that the inequality holds). Let $\varepsilon \in (0,1)$ be small and be determined later. Suppose there exist such $\CP$ and $\CL$ so that $|\CI(\CP,\CL)| \gtrsim N^{\frac{3}{2}-\varepsilon}$. Applying repeatedly the pigeonhole principle to sets of points and lines as in \cite{BKT04}, we obtain subsets $\CP'$ and $\CP''$ of $\CP$, a subset $\CL'$ of $\CL$, and two distinct points $q_0,q_1\in \CP''$ with the following properties
    \begin{itemize}[parsep = 0pt, topsep = 3pt]
        \item For every $q\in \CP'$, $N^{\frac{1}{2}-\varepsilon} \lesssim \mu(q) \lesssim N^{\frac{1}{2}+\varepsilon}$, where $\mu(q) = |\{l\in \CL: q\in l\}|$;
        \item For every $l\in \CL'$, $N^{\frac{1}{2}-\varepsilon} \lesssim \lambda(l) \lesssim N^{\frac{1}{2}+\varepsilon}$, where $\lambda(l) = |\{q\in \CP': q\in l\}|$;
        \item For every $q\in \CP''$, $N^{\frac{1}{2}-\varepsilon} \lesssim \mu'(q) \lesssim N^{\frac{1}{2}+\varepsilon}$, where $\mu'(q) = |\{l\in \CL': q\in l\}|$;
        \item Let $\CL_i = \{l\in \CL': q_i\in l\}$, $i = 0,1$, and $\tilde{\CP} = \{q\in \CP': q\in l_0\cap l_1 \text{ for some } l_0\in \CL_0, l_1\in \CL_1\}$. We have $|\tilde{\CP}| \gtrsim N^{1-C\varepsilon}$ for some suitably chosen constant $C$; namely, the grid-like set $\tilde{\CP}$ covers the majority of points in $\CP$.
    \end{itemize}
    Suppose there exists $l_{0,1}\in \CL'$ passing through both $q_0$ and $q_1$, then since $\lambda(l_{0,1}) \lesssim N^{\frac{1}{2}+\varepsilon}$, we may remove $l_{0,1}$ from $\CL_i$, $i = 0,1$ and points of $\CP'$ on $l_{0,1}$ from $\tilde{\CP}$, and we still have $|\tilde{\CP}|\gtrsim N^{1-C\varepsilon}$ for some suitably chosen $C$ if $\varepsilon$ is sufficiently small. Now for every $l_0\in \CL_0$ and $l_1 \in \CL_1$ so that $l_0\cap l_1 \neq \emptyset$, it is necessary that $\angle (l_0,l_1) = 1$. This allows us to apply Lemma \ref{sepintersec}. Moreover, since each line $l\in \CL_i$ contains $\gtrsim N^{\frac{1}{2}-\varepsilon}$ points in $\CP'$ and $|\CP'| \leq N$, we have $|\CL_i| \lesssim N^{\frac{1}{2}+\varepsilon}$ for $i = 0,1$. This means there are $\lesssim N^{\frac{1}{2}+\varepsilon}$ lines of $\CL'$ passing through $q_0$, and $\lesssim N^{\frac{1}{2}+\varepsilon}$ lines of $\CL'$ passing through $q_1$.\\
    We now consider the following two cases.\\
    \textit{Case 1}: $\|q_0-q_1\| = 1$. Then under the separation condition, different points $q\in \tilde{\CP}$ lie in distinct $p^{-1}$-cubes. Thus, we can safely project onto $(\ZZ/p\ZZ)^2$ without losing any point in $\tilde{\CP}$. Let $\tilde{\CP}_1 = \pi_1(\tilde{\CP})$, $\CL_1 = \pi_1(\CL)$, $\CL_1' = \pi_1(\CL')$, $q_0' = \pi_1(q_0)$ and $q_1' = \pi_1(q_1)$. Then $q_0'$ and $q_1'$ are distinct points in $(\ZZ/p\ZZ)^2$, and the projection is injective on $\CL$. Since the projection preserves the incidences, $\tilde{\CP}_1$ is a subset that can be covered by the intersection points of $\lesssim N^{\frac{1}{2}+\varepsilon}$ lines of $\CL_1'$ passing through $q_0'$ and $\lesssim N^{\frac{1}{2}+\varepsilon}$ lines of $\CL_1'$ passing through $q_1'$. Furthermore, the projection is injective on $\tilde{\CP}$, which implies that $|\tilde{\CP}_1| = |\tilde{\CP}| \gtrsim N^{1-C\varepsilon}$.\\
    \textit{Case 2}: $\|q_0-q_1\| = p^{-j}$ for some $j \geq 1$. Then $q_0$ and $q_1$ lie in the same $p^{-j}$-cube $Q_j$. Lemma \ref{sepintersec} shows that $\tilde{\CP} \subseteq Q_j$.  
    Let $\CP_{k-j} = \iota_{Q_j}(\CP\cap Q_j)$, $\CL_{k-j} = \iota_{Q_j}(\CL \cap Q_j)$, $\CL'_{k-j} = \iota_{Q_j}(\CL'\cap Q_j)$ and $\tilde{\CP}_{k-j} = \iota_{Q_j}(\tilde{\CP}\cap Q_j)$. Let $q_0' = \iota_{Q_j}(q_0)$ and $q_1' = \iota_{Q_j}(q_1)$. Since rescaling is a bijection, $q_0'$ and $q_1'$ are distinct points in $R_{k-j}^2$. Meanwhile, the line segments within $Q_j$ are mapped injectively to lines in $R_{k-j}^2$. Note that rescaling on $Q_j$ also preserves the incidences, so $\tilde{\CP}_{k-j}$ can still be covered by the intersection of $\lesssim N^{\frac{1}{2}+\varepsilon}$ lines of $\CL_{k-j}'$ through $q_0'$ and $\lesssim N^{\frac{1}{2}+\varepsilon}$ lines of $\CL_{k-j}'$ through $q_1'$. In addition, after rescaling we have $\|q_0'-q_1'\| = 1$ in $R_{k-j}^2$ and all lines in $\CL_{k-j}$ are $1$-separated in direction (since the original lines in $\CL$ intersect the same $p^{-j}$-cube, these lines fail to be $1$-separated in distance). So different points of $\CP_{k-j}'$ lie in distinct $p^{-1}$-cubes in $R_{k-j}^2$, which allows us to project onto $(\ZZ/p\ZZ)^2$ without losing distinct points. Let $\CP_1 = \pi_1(\CP_{k-j})$, $\CL_1 = \pi_1(\CL_{k-j})$, $\CL_1' = \pi_1(\CL_{k-j}')$, and $\tilde{\CP}_1 = \pi_1(\tilde{\CP}_{k-j})$. Let $q_0'' = \pi_1(q_0')$ and $q_1'' = \pi_1(q_1')$. Similar to the previous case, we have $|\tilde{\CP}_1| = |\tilde{\CP}_{k-j}| = |\tilde{\CP}| \gtrsim N^{1-C\varepsilon}$. Since the projection preserves the incidences and maps injectively on $\CL_{k-j}$, $\tilde{\CP}_1$ can be covered by the intersection of $\lesssim N^{\frac{1}{2}+\varepsilon}$ lines of $\CL_1'$ through $q_0''$ and $\lesssim N^{\frac{1}{2}+\varepsilon}$ lines of $\CL_1'$ through $q_1''$.\\
    We now proceed with the set of points $\tilde{\CP}_1$ and the set of lines $\CL_1$ in $(\ZZ/p\ZZ)^2$, and we may still denote the two distinct points after rescaling and projection $q_0,q_1\in (\ZZ/p\ZZ)^2$. We can apply a projective transformation mapping $q_0,q_1$ to points at infinity (so the possible line in $\CL_1$ passing through both $q_0$ and $q_1$ is now the line at infinity, which will be discarded anyway) so that the lines in $\CL_1'$ passing through $q_0$ become horizontal lines, while the lines in $\CL_1'$ passing through $q_1$ become vertical lines. Then the lines in $\CL_1'$ through $q_0$ that cover $\tilde{\CP}_1$ are all of the form $\{x = a\}$ for some $a\in \ZZ/p\ZZ$. Let $A\subseteq \ZZ/p\ZZ$ be the set of all such $a$. Similarly, the lines in $\CL_1'$ through $q_1$ that cover $\tilde{\CP}_1$ are all of the form $\{y = b\}$ for some $b\in \ZZ/p\ZZ$. Let $B\subseteq \ZZ/p\ZZ$ be the set of all such $b$. Then the covering condition of $\tilde{\CP}_1$ shows that $|A|,|B| \lesssim N^{\frac{1}{2}+\varepsilon}$. At this point, we finish the step to put $\tilde{\CP}_1$ into a Cartesian product $A\times B$ with $|A|,|B| \lesssim N^{\frac{1}{2}+\varepsilon}$. Note that the injectivity of rescaling and projection on lines and on $\tilde{\CP}$ implies that for each $q\in \tilde{\CP}_1$, we still have 
    \begin{equation*}
        |\{l\in \CL_1: q\in l\}| \gtrsim N^{\frac{1}{2}-\varepsilon}.
    \end{equation*}
    Following the rest of the proof in \cite{BKT04} to extract the additive and multiplicative information on $A$ and $B$ and apply the sum-product estimate, we arrive at the same contradiction.
\end{proof}

\subsection{Proof of Theorem 1.2}
In this section, we prove Theorem \ref{SZp-adicThm}, which extends the result of Theorem 3 in \cite{SZ17} to well-separated lines in $p$-adics. Instead of small sets of points and lines, we now consider sets of points and lines that are of comparable sizes (not necessarily very small), while the lines are still assumed to satisfy the separation condition.\\
We state a $p$-adic analogue of Lemma 8 in \cite{SZ17}, which can be derived directly from the proof in \cite{SZ17}.
\begin{lemma}\label{Cartprod}
    Let $\CP$ be a set of $m$ points of $R_k^2$ and $\CL$ a set of $n$ lines in $R_k^2$ satisfying the separation condition. Fix any constants $c_2 > c_1 > 0$ and suppose that between $c_1K$ and $c_2K$ lines of $\CL$ pass through each point of $\CP$. Assume further that $K \geq 4n/(c_1m)$, $K \geq 8/c_1$ and $K^3\geq 2^6n^2/(c_1^3m)$. Then there are distinct points $q_1,q_2\in \CP$ and a set $G \subseteq \CP\setminus \overline{q_1q_2}$ of cardinality $|G| \geq c_1^4K^4m/(2^9n^2)$ such that $G$ is covered by at most $c_2K$ lines from $\CL$ through $q_1$, and by at most $c_2K$ lines from $\CL$ through $q_2$, where $\overline{q_1q_2}$ is the line in $\CL$ (if exists) containing the points $q_1,q_2\in \CP$.
\end{lemma}

\begin{proof}
    Note that due to the separation condition, if $\overline{q_1q_2} \in \CL$ exists, then it must be unique.\\
    Since the lines of $\CL$ satisfy the separation condition, we are essentially under the Euclidean axiom, where two lines of $\CL$ intersect in at most one point of $\CP$, and two points of $\CP$ lie in at most one line of $\CL$. Therefore, the proof follows the exact same argument as in \cite{SZ17}. 
\end{proof}

Moreover, since the Euclidean axiom still holds for well-separated lines, it is easy to use the same combinatorial argument to obtain the following Cauchy-Schwarz incidence bound as in the Euclidean spaces.
\begin{lemma}\label{CSbd}
    Let $\CP$ be a set of $m$ points in $R_k^2$ and $\CL$ a set of $n$ lines in $R_k^2$ satisfying the separation condition. Then 
    \begin{align*}
        |\CI(\CP,\CL)| & \lesssim mn^{\frac{1}{2}} + n,\\
        |\CI(\CP,\CL)| & \lesssim nm^{\frac{1}{2}} + m.
    \end{align*}
\end{lemma}
Now we proceed with proving Theorem \ref{SZp-adicThm}.
\begin{proof}[Proof of Theorem \ref{SZp-adicThm}]
    As in the proof of Theorem 3 in \cite{SZ17}, we shall prove that there exists a constant $C$ such that for all $\CP$ and $\CL$ with $n^{\frac{7}{8}} < m < n^{\frac{8}{7}}$, $|\CI(\CP,\CL)| < Cm^{\frac{11}{15}}n^{\frac{11}{15}}$. We prove by induction, keeping $n$ fixed and varying $m$.\\
    Base step: For any $m$ such that $n^{\frac{4}{11}} < m < n^{\frac{7}{8}}$, Lemma \ref{CSbd} gives 
    \begin{equation*}
        |\CI(\CP,\CL)| \lesssim mn^{\frac{1}{2}} + n \leq 2m^{\frac{11}{15}}n^{\frac{11}{15}}.
    \end{equation*}
    Induction step: We now assume that for any point set $\CP'$ with $|\CP'| = m'$ and $n^{\frac{7}{8}} < m' < m$, we have $|\CI(\CP',\CL)| \leq C(m')^{\frac{11}{15}}n^{\frac{11}{15}}$. Suppose to the contrary that $|\CI(\CP,\CL)| \geq Cm^{\frac{11}{15}}n^{\frac{11}{15}}$.\\
    Let $I \coloneqq |\CI(\CP,\CL)|$ and $K \coloneqq I/m$. Through repeated applications of Lemma \ref{Cartprod} as in the proof of Theorem 3 in \cite{SZ17}, we obtain sequences $\{A_1,\dots, A_{s+1}\}$ and $\{G_1,\dots, G_s\}$ of subsets of $\CP$ such that 
    \begin{itemize}[parsep = 1pt, topsep = 3pt]
        \item $A_1 = A = \{q\in \CP: \text{there are between $2^{-11}K$ and $2^{15}K$ lines of $\CL$ through $q$}\}$;
        \item $A_1 \supseteq A_2\supseteq \cdots \supseteq A_{s+1}$ with $|A_{s+1}| \leq 2^{-15}m$. For each $1\leq i \leq s$, $G_i \subseteq A_i$ and $A_{i+1} = A_i\setminus G_i$;
        \item $s \leq \frac{m}{\min_i\{|G_i|\}} \leq \frac{2^{68}n^2}{K^4}$, and $|G_i| \geq \frac{c_1^4K^4|A_i|}{2^9n^2} \geq \frac{(2^{-11})^4K^4(2^{-15}m)}{2^9 n^2} \geq \frac{K^4 m}{2^{68}n^2}$ for $\forall\ 1\leq i \leq s$;
        \item For each $i$, there are distinct points $q_{1,i},q_{2,i} \in A_i$ such that $G_i$ is covered by the intersection of at most $2^{15}K$ lines from $\CL$ through $q_{1,i}$ with at most $2^{15}K$ lines from $\CL$ through $q_{2,i}$;
        \item $G_1,\dots, G_s$ contribute at least $I/2$ incidences. In particular, $I \leq 2 \sum\limits_{i=1}^s |\CI(G_i,\CL)|$.
    \end{itemize}
    Now we estimate $|\CI(G_i,\CL)|$. Consider the two special points $q_{1,i}$ and $q_{2,i}$ for $G_i$. Then $G_i$ is covered by the intersections of lines in $\CL$ through $q_{1,i}$ and lines in $\CL$ through $q_{2,i}$. Moreover, since $G_i \cap \overline{q_{1,i}q_{2,i}} = \emptyset$ provided there exists some (unique) line $\overline{q_{1,i}q_{2,i}} \in \CL$ passing through both $q_{1,i}$ and $q_{2,i}$, every point of $G_i$ is the intersection of lines $l_{1,i}$ passing through $q_{1,i}$ and $l_{2,i}$ passing through $q_{2,i}$ with $\angle(l_{1,i},l_{2,i}) = 1$, which allows us to apply Lemma \ref{sepintersec}. We consider the following two cases:\\
    \textit{Case 1}: $\|q_{1,i}-q_{2,i}\| = 1$. Then under the separation condition, every $p^{-1}$-cube contains at most one point of $G_i$. Let $G_i' = \pi_1(G_i)$, $q_{1,i}' = \pi_1(q_{1,i})$, $q_{2,i}' = \pi_1(q_{2,i})$, and $\CL' = \pi_1(\CL)$. Then $q_{1,i}'$ and $q_{2,i}'$ are distinct points in $(\ZZ/p\ZZ)^2$. Since the projection preserves the incidences, and the lines of $\CL$ satisfy the separation condition, $G_i'$ is still covered by at most $2^{15}K$ lines of $\CL'$ through $q_{1,i}'$ and by at most $2^{15}K$ lines of $\CL$ through $q_{2,i}'$. Moreover, we have $|G_i'| = |G_i|$, $|\CL'| = |\CL|$ and $|\CI(G_i,\CL)| = |\CI(G_i',\CL')|$ (since every point of $G_i$ is an intersection of two lines of $\CL$, and two lines of $\CL$ intersect in at most one point).\\
    \textit{Case 2}: $\|q_{1,i}-q_{2,i}\|  = p^{-j}$ for some $j \geq 1$. Then $q_{1,i}$ and $q_{2,i}$ lie in the same $p^{-j}$-cube $Q_j$. This means the lines in $\CL$ through $q_{1,i}$ and through $q_{2,i}$ are all $1$-separated in direction. Lemma \ref{sepintersec} shows that $G_i \subseteq Q_j$.\\
    Let $\tilde{G}_i = \iota_{Q_j}(G_i\cap Q_j)$, $\tilde{\CL}= \iota_{Q_j}(\CL_j)$, where $\CL_j$ is the set of intersections of lines in $\CL$ with $Q_j$. Let $\tilde{q}_{1,i} = \iota_{Q_j}(q_{1,i})$ and $\tilde{q}_{2,i} = \iota_{Q_j}(q_{2,i})$. Since rescaling is a bijection, $\tilde{q}_{1,i}$ and $\tilde{q}_{2,i}$ are distinct points in $R_{k-j}^2$ and $|\tilde{G}_i| = |G_i|$, $|\tilde{\CL}| = |\CL_j|$. Note that rescaling on $Q_j$ also preserves the incidences, so $\tilde{G}_i$ is covered by at most $2^{15}K$ lines in $\tilde{\CL}$ through $\tilde{q}_{1,i}$ and by at most $2^{15}K$ lines in $\tilde{\CL}$ through $\tilde{q}_{2,i}$, and we have $|\CI(\tilde{G}_i,\tilde{\CL})| = |\CI(G_i,\CL)|$. In addition, we have $\|\tilde{q}_{1,i}-\tilde{q}_{2,i}\| = 1$ in $R_{k-j}^2$ and all lines in $\tilde{\CL}$ are still $1$-separated in direction. Thus, each $p^{-1}$-cube in $R_{k-j}^2$ contains at most one point of $\tilde{G}_i$.\\ 
    Similar to Case 1, let $G_i' = \pi_1(\tilde{G}_i)$, $\CL' = \pi_1(\tilde{\CL})$, $q_{1,i}' = \pi_1(\tilde{q}_{1,i})$ and $q_{2,i}' = \pi_1(\tilde{q}_{2,i})$. Since projection preserves the incidences, and the lines of $\tilde{\CL}$ satisfy the separation condition, $G_i'$ is still covered by at most $2^{15}K$ lines of $\CL'$ through $q_{1,i}'$ and by at most $2^{15}K$ lines of $\CL$ through $q_{2,i}'$. Moreover, $|G_i'| = |\tilde{G}_i| = |G_i|$, $|\CL'| = |\tilde{\CL}| = |\CL|$ and $|\CI(G_i',\CL')| = |\CI(\tilde{G}_i,\tilde{\CL})| = |\CI(G_i,\CL)|$.\\
    Now it suffices to consider the corresponding $G_i'$, $\CL'$, $q_{1,i}'$ and $q_{2,i}'$ in $(\ZZ/p\ZZ)^2$, which allows us to use projective transformations and apply Theorem 4 in \cite{SZ17}. Therefore, arguing as the proof in \cite{SZ17}, we can choose $C$ sufficiently large, independent of $n$ and $m$, leading to a contradiction. 
\end{proof}

\subsection{Summary of incidence estimates for well-separated lines}
Recall that since we assume that the lines of $\CL$ satisfy the separation condition, we have the Cauchy-Schwarz bounds Lemma \ref{CSbd}:
\begin{align*}
    |\CI(\CP,\CL)| & \lesssim |\CP|^{\frac{1}{2}}|\CL| + |\CP|;\\
    |\CI(\CP,\CL)| & \lesssim |\CP||\CL|^{\frac{1}{2}} + |\CL|.
\end{align*}
In particular, when $|\CP| \gg |\CL|^2$, the dominated term would be $|\CP|$, so for all large sets $|\CP|$ (which appears very generally for large $k$), the incidence bound becomes $|\CI(\CP,\CL)| \lesssim |\CP|$.\\
Combined with Theorem \ref{SZp-adicThm}, we have the same table of upper bounds in $p$-adics as in finite fields:
\begin{center}
\begin{tabular}{ |c c c| } 
 \hline
 Range of $|\CL|$ & \qquad \qquad & Best bound\\
 \hline
$|\CL| < |\CP|^{\frac{1}{2}}$ & \qquad \qquad & $\lesssim |\CP|$ \\ 
 $|\CP|^{\frac{1}{2}} < |\CL| < |\CP|^{\frac{7}{8}}$ & \qquad \qquad &  $\lesssim |\CP|^{\frac{1}{2}}|\CL|$ \\ 
  $|\CP|^{\frac{7}{8}} < |\CL| < |\CP|^{\frac{8}{7}}$ & \qquad \qquad &  $\lesssim |\CP|^{\frac{11}{15}}|\CL|^{\frac{11}{15}}$\\ 
  $|\CP|^{\frac{8}{7}} < |\CL| < |\CP|^2$ & \qquad \qquad &  $\lesssim |\CP||\CL|^{\frac{1}{2}}$\\
  $|\CP|^2 < |\CL| $ & \qquad \qquad &  $\lesssim |\CL|$\\
 \hline
\end{tabular}
\end{center}

\smallskip

\section{Incidences for non-separated lines}
We now discuss more general point-line incidences in the $p$-adic setting, where the Euclidean axiom is violated and thus two lines can intersect at more than one point. We pay special attention to incidences for sets of points and lines that possess dimensional structures, in which case we say that these sets satisfy certain dimensional spacing conditions.

\subsection{Simple dimensional spacing conditions}
We first introduce a simple version of dimensional spacing conditions concerning only sets of unweighted points and lines, as motivation for the further generalization in the next section. These conditions provide additional information on the distributions of points and lines, which allows us to estimate the incidences in a more delicate manner.
\begin{definition}
    Let $\alpha,\beta \in [0,2]$. A set $\CP$ of points in $R_k^2$ is called an $\alpha$-set if for every $0\leq j \leq k$ and for every $p^{-j}$-cube $Q_j$, 
    \begin{equation*}
        |\CP\cap Q_j| \leq p^{\alpha(k-j)}.
    \end{equation*}
    A set $\CL$ of lines in $R_k^2$ is called an $\beta$-set if for every $0\leq j \leq k$, and for every $p^{-j}$-tube $T_j$, 
    \begin{equation*}
        |\{l\in \CL: l\subseteq T_j\}| \leq p^{\beta(k-j)}.
    \end{equation*}
\end{definition}
\begin{remark}
    We see that every $p^{-j}$-cube $Q_j$ contains $p^{2(k-j)}$ points of $R_k^2$, and every $p^{-j}$-tube contains $p^{2(k-j)}$ lines in $R_k^2$. The uniformity among cubes or tubes endows the set with a dimensional structure.
\end{remark}

Now let $\CL$ be a set of lines in $R_k^2$. Define for each $l\in \CL$
\begin{equation*}
    \CL_j(l) = \{l'\in \CL: l'\cap l \neq \emptyset,\ \angle(l,l') = p^{-j}\} \quad \text{for}\ \ 0 \leq j \leq k.
\end{equation*}
In particular, if $\CL$ is an $\beta$-set of lines, then by definition, we have 
\begin{equation*}
    |\CL_j(l)| \leq p^{\beta(k-j)} \quad \text{for every $l\in \CL$},
\end{equation*}
because $\angle(l,l') = p^{-j}$ and $l\cap l' \neq \emptyset$ imply that $l$ and $l'$ lie in the same $p^{-j}$-tube.\\
Let $\CP$ be a set of points of $R_k^2$. Define for each $q\in \CP$ 
\begin{equation*}
    \CL(q) = \{l\in \CL: q\in l\}.
\end{equation*}

The following lemma can be easily deduced from the basic $p$-adic geometry.
\begin{lemma}\label{intersectsize}
    Let $\CP$ be an $\alpha$-set of points of $R_k^2$, $\CL$ an $\beta$-set of lines in $R_k^2$. Fix $l\in \CL$. For every $0 \leq j \leq k$, we have 
    \begin{equation*}
        \sum\limits_{q\in l} |\CL(q)\cap \CL_j(l)| \leq p^{aj}|\CL_j(l)|,
    \end{equation*}
    where $a = \min\{\alpha,1\}$.
\end{lemma}
An incidence estimate can also be attained following the same combinatorial argument in Section 3 of \cite{FR24} using Lemma \ref{intersectsize}. We state it below and omit the proof.
\begin{proposition}\label{alphabeta}
    Let $\CP$ be an $\alpha$-set of points in $R_k^2$, and $\CL$ be a $\beta$-set of lines in $R_k^2$.
    \begin{enumerate}[label = (\alph*), parsep = 0pt, topsep = 2pt]
        \item Let $a = \min(\alpha,1)$ and assume $a \geq \beta$. Then 
        \begin{equation*}
            |I(\CP,\CL)| \leq p^{k\frac{a\beta}{a+\beta}} (k+1)^{\frac{a}{a+\beta}}|\CP|^{\frac{a}{a+\beta}}|\CL|^{\frac{\beta}{a+\beta}}.
        \end{equation*}
        \item Let $b = \min(\beta,1)$ and assume $b \geq \alpha$. Then 
        \begin{equation*}
            |I(\CP,\CL)| \leq p^{k\frac{\alpha b}{\alpha+b}} (k+1)^{\frac{b}{b+\alpha}}|\CP|^{\frac{\alpha}{\alpha+b}}|\CL|^{\frac{b}{\alpha+b}}.
        \end{equation*}
    \end{enumerate}
\end{proposition}

\subsection{Generalized incidences and induction on scales}
In this section, we start with the basic Fourier analysis in the $p$-adic setting. Then we define formally the incidence for sets of weighted points and lines and prove an intermediate result which lays the foundation for the inductive step in the proof of Theorem \ref{weighted}. Finally, we generalize the dimensional spacing conditions to sets of weighted points and lines and prove Theorem \ref{weighted}.
\subsubsection{Basic Fourier analysis}
We consider the following Fourier transform for functions on $R_k^2$: Let $f: R_k^2 \rightarrow \CC$ be a function. The Fourier transform of $f$, denoted by $\widehat{f}$, is given by 
\begin{equation*}
    \widehat{f}(\xi) = p^{-k}\sum\limits_{x\in R_k^2} f(x)e^{-\frac{2\pi i \langle \xi, x\rangle}{p^k}}.
\end{equation*}
Using this normalized Fourier transform, we have the following Plancherel's and Parseval's identities:
\begin{align*}
    \sum\limits_{x\in R_k^2} f(x)\overline{g(x)}  = \sum\limits_{\xi \in R_k^2} \widehat{f}(\xi) \overline{\widehat{g}(\xi)};\quad 
    \sum\limits_{x\in R_k^2} |f(x)|^2   = \sum\limits_{\xi \in R_k^2} |\hat{f}(\xi)|^2.
\end{align*}
The convolution property now becomes 
\begin{align*}
    \widehat{\varphi * \psi}(\xi) & = p^{-k}\sum\limits_{x\in R_k^2} (\varphi * \psi)(x)e^{-\frac{2\pi i \langle \xi,x\rangle}{p^k}}\\
    & = p^{-k} \sum\limits_{x\in R_k^2} \sum\limits_{y\in R_k^2} \varphi(x-y)\psi(y) e^{-\frac{2\pi i \langle \xi,y\rangle}{p^k}}\cdot e^{-\frac{2\pi i \langle \xi,x-y\rangle}{p^k}}\\
    & = p^{-k} \sum\limits_{y\in R_k^2} \psi(y)e^{-\frac{2\pi i \langle \xi,y\rangle}{p^k}} \sum\limits_{x\in R_k^2}  \varphi(x-y)e^{-\frac{2\pi i \langle \xi,x-y\rangle}{p^k}}\\
    & = p^k \widehat{\varphi}(\xi) \widehat{\psi}(\xi).
\end{align*}

If we define $\chi_x(\xi) = e^{-\frac{2\pi i \langle \xi,x\rangle}{p^k}}$ for each $x\in R_k^2$, and let $M$ be a $R_k$-submodule of $R_k^2$, then we have 
\begin{align*}
    \widehat{\mathbf{1}_M}(\xi) = C\mathbf{1}_{M^{\sharp}}(\xi),
\end{align*}
where $C$ is some constant, and $M^{\sharp}$ is the annihilator of $M$ in $R_k^2$, defined as 
\begin{equation*}
    M^{\sharp} = \{x\in R_k^2: \forall\ \xi \in M,\ \chi_x(\xi) = 1\}.
\end{equation*}

The orthogonal to $M$ is defined by 
\begin{equation*}
    M^{\perp} = \{y\in R_k^2: \forall\ x\in M, \langle x,y\rangle = 0\}.
\end{equation*}
The orthogonal to $M$ and annihilator of $M$ coincide provided certain conditions are met:
\begin{lemma}
    Let $Q$ be a commutative ring possessing a generating character and $A$ a $Q$-submodule of $Q^d$. Then $A^{\sharp} = A^{\perp}$.
\end{lemma}
In particular, for $Q = R_k = \ZZ/p^k\ZZ$, we have $M^{\sharp} = M^{\perp}$.

\subsubsection{Incidences for weighted points and lines}
From basic $p$-adic geometry, in $R_k^2$ a point is a $p^{-k}$-cube and a line is a $p^{-k}$-tube. We observe that for any $1\leq j \leq k-1$, each $p^{-j}$-cube in $R_k^2$ can be identified with a point in $R_j^2$, and each $p^{-j}$-tube can be identified with a line in $R_j^2$. As a result, the incidence between $p^{-j}$-cubes and $p^{-j}$-tubes in $R_k^2$ is equivalent to the incidence between points and lines in $R_j^2$. These identifications enable us to consider multiple scales for the point-line incidence problem and carry out the induction-on-scales argument derived from the Euclidean setting.\\
Furthermore, we notice that in $p$-adics, the concepts of "cube" and "ball" are identical. A consequence of the equivalence of cubes and balls is that every point in the same cube will be thickened to be exactly that cube, so that we always end up with duplicated cubes and tubes after thickening. To deal with the duplication, we now extend the definition of incidence to sets of weighted points and lines, which is an analogue of the set-up in \cite{Bradshaw23} in the $p$-adic setting.

\begin{definition}
    Let $\CP$ be a set of weighted points in $R_k^2$ with weight function $w: \CP \mapsto \NN$ and $\CL$ a set of weighted lines in $R_k^2$ with weight function $\omega: \CL \mapsto \NN$. We define the generalized incidence counting function
    \begin{equation*}
        I_w(\CP,\CL) \coloneqq \sum\limits_{q\in \CP} \sum\limits_{l \in \CL} w(q)\omega(l) \mathbf{1}_{\{q\in l\}}
    \end{equation*}
\end{definition}
\begin{remark}
    We see that 
    \begin{equation*}
        I_w(\CP,\CL) \leq \left(\sum\limits_{q\in \CP} w(q)\right)\left(\sum\limits_{l\in \CL} \omega(l)\right) \eqqcolon |\CP|_w|\CL|_w.
    \end{equation*}
\end{remark}

The general principle of weighted points and lines is as follows. After thickening, some points may become identical cubes, in which case we consider them a single cube (alternatively, a single point at some lower scale) with an associated weight. If the original points are already weighted, then their weights sum if they become the same cube after thickening. Similar rules apply to tubes thickened from lines.\\
Now we prove our first result, which will be the foundation for our inductive argument. The proof is based on the standard high-low method inspired by \cite{GSW19}. We refer to \cite{FGR22}, \cite{Bradshaw23}, \cite{FR24}, \cite{OS23}, and \cite{RW23} for more applications of the high-low method to tube-ball incidences as well as the Furstenberg set problem in the Euclidean setting.

\begin{proposition}\label{highlowlemma}
    Let $\CP$ be a set of distinct weighted points of $R_k^2$ with weight function $w$, and let $\CL$ be a set of distinct (unweighted) lines in $R_k^2$. Let $1\leq j \leq k-1$. Then
    \begin{equation*}
        I_w(\CP,\CL) \leq p^{\frac{k+j-1}{2}}|\CL|^{\frac{1}{2}}\left(\sum\limits_{q\in \CP} w(q)^2\right)^{\frac{1}{2}} + p^{-j}I_w(\CP_{k-j},\CL_{k-j}),
    \end{equation*}
    where $\CP_{k-j}$ and $\CL_{k-j}$ are, respectively, the sets of weighted $p^{-(k-j)}$-cubes and $p^{-(k-j)}$-tubes formed by thickening $\CP$ and $\CL$ by a factor of $p^j$.
\end{proposition}

\begin{proof}
    Let $f(x) = \sum\limits_{q\in \CP} w(q)\mathbf{1}_{\{q\}}(x)$, and $g(x) = \sum\limits_{l\in \CL} \mathbf{1}_{l}(x)$. Then by definition and Plancherel,  
    \begin{equation*}
        I_w(\CP,\CL) = \sum\limits_{x\in R_k^2} f(x)g(x) = \sum\limits_{\xi\in R_k^2} \hat{f}(\xi)\overline{\hat{g}(\xi)}.
    \end{equation*}
    Let $\eta = \mathbf{1}_{Q_j^{(0)}}$, where $Q_j^{(0)}$ is the $p^{-j}$-cube in $R_k^2$ containing $(0,0)$. Then we can decompose the sum into \textit{high frequency term} and \textit{low frequency term} as follows.
    \begin{align*}
        I_w(\CP,\CL) = \sum\limits_{\xi\in R_k^2} \hat{f}(\xi)\overline{\hat{g}(\xi)}\eta(\xi) + \sum\limits_{\xi\in R_k^2} \hat{f}(\xi)\overline{\hat{g}(\xi)}(1-\eta(\xi)) \eqqcolon L + H.
    \end{align*}
    \textit{High frequency case}. We see that by Cauchy-Schwarz inequality, 
    \begin{align*}
        H & = \sum\limits_{\xi\in R_k^2} \hat{f}(\xi)\overline{\hat{g}(\xi)}(1-\eta(\xi))\\
        & \leq \left(\sum\limits_{\xi\in R_k^2} |\hat{f}(\xi)|^2\right)^{\frac{1}{2}}\left(\sum\limits_{\xi\in R_k^2} |\hat{g}(\xi)|^2(1-\eta(\xi))^2\right)^{\frac{1}{2}}.
    \end{align*}
    Note that by Parseval's identity, 
    \begin{equation*}
        \sum\limits_{\xi\in R_k^2} |\hat{f}(\xi)|^2 = \sum\limits_{x\in R_k^2} |f(x)|^2 = \sum\limits_{q\in \CP} w(q)^2.
    \end{equation*}
    It is easy to see that for every line $l$ in $R_k^2$, $\widehat{\mathbf{1}_{l}}$ is supported on the dual line $l'$ through $(0,0)$ perpendicular to $l$. Let $\xi \in R_k^2\setminus Q_j^{(0)}$. Then $\|\xi - 0\| \geq p^{-(j-1)}$, and so there is at most $p^{j-1}$ \textit{different directions} for lines passing through both $0$ and $\xi$. Also, parallel lines are disjoint and correspond to the same dual line. So 
    \begin{align*}
        \sum\limits_{\xi\in R_k^2} |\hat{g}(\xi)|^2(1-\eta(\xi))^2 & = \sum\limits_{\xi\in R_k^2} (1-\eta(\xi))^2\left |\sum\limits_{l\in \CL} \widehat{\mathbf{1}_{l}}(\xi)\right|^2\\
        & \leq p^{j-1}\sum\limits_{\xi\in R_k^2} (1-\eta(\xi))^2 \sum\limits_{b\in \PP R_k^1} \left|\sum\limits_{l\parallel b}\widehat{\mathbf{1}_{l}}(\xi)\right|^2\\
        & \leq p^{j-1}\sum\limits_{b\in \PP R_k^2} \sum\limits_{\xi \in R_k^2} \left|\sum\limits_{l\parallel b}\widehat{\mathbf{1}_{l}}(\xi)\right|^2\\
        & = p^{j-1}\sum\limits_{b\in \PP R_k^2} \sum\limits_{x\in R_k^2} \left|\sum\limits_{l\parallel b}\mathbf{1}_{l}(x)\right|^2\\
        & \leq p^{j-1} \sum\limits_{b\in \PP R_k^2} p^k|\{l\in \CL: l\parallel b\}|\\
        & \leq p^{k+j-1}|\CL|.
    \end{align*}
    It follows that 
    \begin{equation*}
        H \leq \left(\sum\limits_{q\in \CP} w(q)^2\right)^{\frac{1}{2}} p^{\frac{k+j-1}{2}}|\CL|^{\frac{1}{2}}.
    \end{equation*}
    \textit{Low frequency case}. Note that $\eta = \overline{\eta} = |\eta|^2$. So we have 
    \begin{align*}
        L & = \sum\limits_{\xi \in R_k^2} \hat{f}(\xi)\overline{\hat{g}(\xi)}|\eta|^2(\xi) = \sum\limits_{\xi \in R_k^2} \hat{f}(\xi)\eta(\xi)\overline{\hat{g}(\xi)\eta(\xi)} = p^{-2k} \sum\limits_{x \in R_k^2} (f * h)(x)(g * h)(x),
    \end{align*}
    where $h$ satisfies that $\hat{h} = \eta$. Since each $Q_j^{(0)}$ is an $R_k$-submodule of $R_k^2$, the discussion in the previous part suggests that if $\hat{h} = \eta$, then $h = C\mathbf{1}_{Q_{k-j}^{(0)}}$, where $Q_{k-j}^{(0)}$ is the $p^{-(k-j)}$-cube in $R_k^2$ containing $(0,0)$. We can explicitly calculate that 
    \begin{equation*}
        h = p^{k-2j}\mathbf{1}_{Q_{k-j}^{(0)}}.
    \end{equation*}
    We now study the effect of convolution with $h$. For every $x\in R_k^2$,
    \begin{align*}
        (f*h)(x) & = \sum\limits_{y\in R_k^2} f(y)h(x-y) = p^{k-2j}\sum\limits_{y\in R_k^2} \sum\limits_{q\in \CP} w(q)\mathbf{1}_{\{q\}}(y) \mathbf{1}_{Q_{k-j}^{(0)}}(x-y)\\
        & = p^{k-2j}\sum\limits_{q\in \CP} w(q) \sum\limits_{y\in R_k^2}\mathbf{1}_{\{q\}}(y) \mathbf{1}_{Q_{k-j}^{(x)}}(y) = p^{k-2j}\sum\limits_{q\in \CP\cap Q_{k-j}^{(x)}} w(q),
    \end{align*}
    and
    \begin{align*}
        (g*h)(x) & = \sum\limits_{y\in R_k^2} g(y)h(x-y) = p^{k-2j}\sum\limits_{y\in R_k^2} \sum\limits_{l\in \CL} \mathbf{1}_{l}(y)\mathbf{1}_{Q_{k-j}^{(0)}}(x-y)\\
        & = p^{k-2j} \sum\limits_{l\in \CL} \sum\limits_{y\in R_k^2} \mathbf{1}_{l}(y)\mathbf{1}_{Q_{k-j}^{(x)}}(y) = p^{k-2j}\sum\limits_{l\in \CL} |l\cap Q_{k-j}^{(x)}|,
    \end{align*}
    where $Q_{k-j}^{(x)}$ is the $p^{-(k-j)}$-cube in $R_k^2$ containing $x$. Thus, we have 
    \begin{align*}
        L & = p^{-2k}\cdot p^{2(k-2j)} \sum_{x\in R_k^2} \left(\sum_{q\in \CP\cap Q_{k-j}^{(x)}} w(q) \right) \left(\sum_{l\in \CL} |l\cap Q_{k-j}^{(x)}|\right)\\
        & = p^{-4j} \sum_{x\in R_k^2} \sum_{q\in \CP\cap Q_{k-j}^{(x)}} w(q) \sum_{l: l\cap Q_{k-j}^{(x)}\neq \emptyset} |l\cap Q_{k-j}^{(x)}|\\
        & = p^{-4j}\cdot p^j \sum_{x\in R_k^2} \left(\sum_{q\in \CP\cap Q_{k-j}^{(x)}} w(q)\right) |\{l\in \CL: l\cap Q_{k-j}^{(x)} \neq \emptyset\}|\\
        & = p^{-3j}\cdot p^{2j}\sum_{Q_{k-j}} \left(\sum_{q\in \CP\cap Q_{k-j}} w(q)\right) |\{l\in \CL: T_{k-j}^{(l)}\cap Q_{k-j} \neq \emptyset\}|\\
        & = p^{-j}\sum_{Q_{k-j}\in \CP_{k-j}} \left(\sum_{q\in \CP\cap Q_{k-j}} w(q)\right)\left( \sum_{T_{k-j}\in \CL_{k-j}} |\{l\in \CL: l\subseteq T_{k-j},\ T_{k-j}\cap Q_{k-j} \neq \emptyset\}| \right)\\
        & = p^{-j}I_w(\CP_{k-j},\CL_{k-j}),
    \end{align*}
    where $T_{k-j}^{(l)}$ is the $p^{-(k-j)}$-tube containing $l$. Here we use the fact that there are $p^{2j}$ points in $R_k^2$ determining the same $p^{-(k-j)}$-cube and the general principle introduced above. Therefore, we conclude that 
    \begin{align*}
        I_w(\CP,\CL) \leq p^{\frac{k+j-1}{2}}|\CL|^{\frac{1}{2}}\left(\sum\limits_{q\in \CP} w(q)^2\right)^{\frac{1}{2}} + p^{-j}I_w(\CP_{k-j},\CL_{k-j}).
    \end{align*}
\end{proof}

\subsubsection{Generalized dimensional conditions and Proof of Theorem 1.3}
We now generalize the dimensional spacing conditions to sets of weighted cubes and tubes.
\begin{definition}
    We say that a family $\CQ$ of weighted $p^{-j}$-cubes in $R_k^2$, with weight function $w$, is a $(p^{-j},\alpha,K_{\alpha})$-set if for every $0\leq \ell \leq j$ and for every $p^{-\ell}$-cube $Q_{\ell}$, we have 
    \begin{equation*}
        \sum\limits_{\substack{q\in \CQ\\ q\subseteq Q_{\ell}}} w(q) \leq K_{\alpha}p^{\alpha(j-\ell)}.
    \end{equation*}
    In particular, when $\ell = j$, the condition implies that $w(q) \leq K_{\alpha}$ for every $q\in \CQ$.\\
    Similarly, we say that a family $\CT$ of weighted $p^{-j}$-tubes in $R_k^2$, with weight function $\omega$, is a $(p^{-j},\beta,K_{\beta})$-set if for every $0\leq \ell \leq j$ and for every $p^{-\ell}$-tube $T_{\ell}$, we have 
    \begin{equation*}
        \sum\limits_{\substack{t\in \CT \\ t\subseteq T_{\ell}}} \omega(t) \leq K_{\beta}p^{\beta(j-\ell)}.
    \end{equation*}
\end{definition}
After extending the conditions to weighted sets, the behavior under thickening becomes more consistent compared to when we only restrict to unweighted sets. Namely, thickening a $(p^{-m},\alpha,K_{\alpha})$-set of cubes or tubes by a factor of $p^r$ for some $1\leq r \leq m-1$ yields a $(p^{-(m-r)},\alpha,p^{\alpha r} K_{\alpha})$-set of cubes or tubes.\\
For a weighted set $\CA$ with weight function $w$, define 
\begin{equation*}
    |\CA|_w = \sum\limits_{a\in \CA} w(a).
\end{equation*}
When $\CA$ is an unweighted set, i.e. $w(a) \equiv 1$ for all $a\in \CA$, we have $|\CA|_w = |\CA|$, the cardinality of $\CA$.

\begin{proposition}\label{inductlemma}
    Let $\CP$ be a set of weighted points of $R_k^2$ with weight function $w$, and let $\CL$ be a set of weighted lines in $R_k^2$, with weight function $\omega$. Suppose that $\max_{q\in \CP} w(q) \leq K_{\alpha}$ and $\max_{l\in \CL} \omega(l) \leq K_{\beta}$ for some positive constants $K_{\alpha}$ and $K_{\beta}$. Let $1\leq j \leq k-1$. Then
    \begin{equation*}
        I_w(\CP,\CL) \leq p^{\frac{k+j-1}{2}}(K_{\alpha}K_{\beta})^{\frac{1}{2}}|\CP|_w^{\frac{1}{2}}|\CL|_w^{\frac{1}{2}} + p^{-j}I_w(\CP_{k-j},\CL_{k-j}),
    \end{equation*}
    where $\CP_{k-j}$ and $\CL_{k-j}$ are, respectively, the weighted sets of $p^{-(k-j)}$-cubes and $p^{-(k-j)}$-tubes formed by thickening $\CP$ and $\CL$ by a factor of $p^j$.
\end{proposition}
\begin{proof}
    We partition $\CP$ into $M_1$ unweighted groups $\CP^1,\dots, \CP^{M_1}$ and partition $\CL$ into $M_2$ unweighted groups $\CL^1,\dots, \CL^{M_2}$, with $M_1\leq K_{\alpha}$ and $M_2\leq K_{\beta}$. For each pair of unweighted sets $(\CP^i,\CL^r)$ of points and lines, applying Proposition \ref{highlowlemma} with the weight function $w \equiv 1$ we obtain 
    \begin{equation*}
        I(\CP^i,\CL^r) \leq p^{\frac{k+j-1}{2}}|\CP^i|^{\frac{1}{2}}|\CL^r|^{\frac{1}{2}} + p^{-j}I_w(\CP_{k-j}^i,\CL_{k-j}^r).
    \end{equation*}
    We observe that $\CP_{k-j} = \bigcup_{i=1}^{M_1} \CP_{k-j}^i$ and $\CL_{k-j} = \bigcup_{r=1}^{M_2} \CL_{k-j}^r$ as sets. The general principle indicates that 
    \begin{align*}
        & w(q) = \sum_{i: q\in \CP_{k-j}^i} w_i(q) \quad \text{for every } q\in \CP_{k-j},\\
        & \omega(t) = \sum_{r: t\in \CL_{k-j}^r} \omega_r(t) \quad \text{for every } t\in \CL_{k-j}.
    \end{align*}
    It follows that 
    \begin{align*}
        \sum_{i=1}^{M_1} \sum_{r=1}^{M_2} I_w(\CP_{k-j}^i,\CL_{k-j}^r) & = \sum_{i=1}^{M_1} \sum_{r=1}^{M_2} \sum_{q\in \CP_{k-j}^i} \sum_{t\in \CL_{k-j}^r} w_i(q)\omega_r(t) \mathbf{1}_{\{q\cap t \neq \emptyset\}}\\
        & = \sum_{q\in \CP_{k-j}} \sum\limits_{t\in \CL_{k-j}} \sum\limits_{i: q\in \CP_{k-j}^i} \sum\limits_{r: t\in \CL_{k-j}^r} w_i(q)\omega_r(t) \mathbf{1}_{\{q\cap t \neq \emptyset\}}\\
        & = \sum\limits_{q\in \CP_{k-j}}\sum\limits_{t\in \CL_{k-j}} \mathbf{1}_{\{q\cap t \neq \emptyset\}} \left(\sum\limits_{i: q\in \CP_{k-j}^i} w_i(q)\right) \left(\sum\limits_{r: t\in \CL_{k-j}^r}\omega_r(t)\right)\\
        & = \sum\limits_{q\in \CP_{k-j}}\sum\limits_{t\in \CL_{k-j}} w(q)\omega(t) \mathbf{1}_{\{q\cap t \neq \emptyset\}} = I_w(\CP_{k-j},\CL_{k-j}).
    \end{align*}
    Combining the contributions of incidences from each pair of unweighted sets gives that
    \begin{align*}
        \sum_{i=1}^{M_1} \sum_{r=1}^{M_2} I(\CP^i,\CL^r) = \sum_{q'\in \CP}|\{i: q'\in\CP^i\}| \sum_{l\in \CL} |\{r: l\in \CL^r\}| \mathbf{1}_{\{q'\cap l \neq \emptyset\}} = \sum_{q'\in \CP} w(q) \sum_{l\in \CL} \omega(l) \mathbf{1}_{\{q'\cap l \neq \emptyset\}}.
    \end{align*}
    So we have  
    \begin{align*}
        I_w(\CP,\CL) & = \sum_{q'\in \CP} \sum_{l\in \CL}  w(q) \omega(l) \mathbf{1}_{\{q'\cap l \neq \emptyset\}} = \sum\limits_{i=1}^{M_1} \sum\limits_{r=1}^{M_2} I(\CP^i,\CL^r)\\
        & \leq \sum\limits_{i=1}^{M_1} \sum\limits_{r=1}^{M_2}\left( p^{\frac{k+j-1}{2}}|\CP^i|^{\frac{1}{2}}|\CL^r|^{\frac{1}{2}} + p^{-j}I_w(\CP_{k-j}^i,\CL_{k-j}^r)\right)\\
        & \leq p^{\frac{k+j-1}{2}} M_1^{\frac{1}{2}}\left(\sum\limits_i |\CP^i|\right)^{\frac{1}{2}}M_2^{\frac{1}{2}}\left(\sum\limits_r |\CL^r|\right)^{\frac{1}{2}} + p^{-j}I_w(\CP_{k-j},\CL_{k-j})\\
        & \leq p^{\frac{k+j-1}{2}}(K_{\alpha}K_{\beta})^{\frac{1}{2}}|\CP|_w^{\frac{1}{2}}|\CL|_w^{\frac{1}{2}} + p^{-j}I_w(\CP_{k-j},\CL_{k-j}).
    \end{align*}
    Here we use the simple observation that 
    \begin{equation*}
        \sum_i |\CP^i| = |\CP|_w,\quad \text{and} \quad \sum_r |\CL^r| = |\CL|_w.
    \end{equation*}
\end{proof}

At this point, we are ready to prove Theorem \ref{weighted}, which is a direct extension of Theorem 1.5 in \cite{FR24}. 
\begin{proof}[Proof of Theorem \ref{weighted}]
    We prove by induction on $k$. The base case will be $k = 1$. We see that by definition, $|\CP|_w \leq K_{\alpha}p^{\alpha}$ and $|\CL|_w\leq K_{\beta}p^{\beta}$. Then 
    \begin{align*}
        I_w(\CP,\CL)& \leq |\CP|_w|\CL|_w \leq p^{c(\alpha+\beta)}(K_{\alpha}K_{\beta})^c|\CP|_w^{1-c}|\CL|_w^{1-c}\\
        & \leq C(p,\varepsilon)p^{(c+\varepsilon)}(K_{\alpha}K_{\beta})^c |\CP|_w^{1-c}|\CL|_w^{1-c}.
    \end{align*}
    For the induction step, we assume that the result holds for scales with $j \leq k-1$. 
    If $|\CP|_w|\CL|_w \leq p^kK_{\alpha}K_{\beta}$, then we have 
    \begin{align*}
        I_w(\CP,\CL)& \leq |\CP|_w|\CL|_w \leq p^{kc}(K_{\alpha}K_{\beta})^c|\CP|_w^{1-c}|\CL|_w^{1-c}\\
        & \leq C(p,\varepsilon)p^{k(c+\varepsilon)}(K_{\alpha}K_{\beta})^c|\CP|_w^{1-c}|\CL|_w^{1-c}.
    \end{align*}
    So we may now assume that $|\CP|_w|\CL|_w \geq p^kK_{\alpha}K_{\beta}$. Note that $\CP$ and $\CL$ satisfy the conditions in Proposition \ref{inductlemma} with constants $K_{\alpha}$ and $K_{\beta}$. Applying Proposition \ref{inductlemma} with $j = 1$, we obtain 
    \begin{align*}
        I_w(\CP,\CL)\leq p^{\frac{k}{2}}(K_{\alpha}K_{\beta})^{\frac{1}{2}}|\CP|_w^{\frac{1}{2}}|\CL|_w^{\frac{1}{2}} + p^{-1}I_w(\CP_{k-1},\CL_{k-1}).
    \end{align*}
    Since $|\CP|_w|\CL|_w \geq p^kK_{\alpha}K_{\beta}$, we have 
    \begin{equation*}
        p^{\frac{k}{2}}(K_{\alpha}K_{\beta})^{\frac{1}{2}}|\CP|_w^{\frac{1}{2}}|\CL|_w^{\frac{1}{2}} \leq p^{kc}(K_{\alpha}K_{\beta})^{c} |\CP|_w^{1-c}|\CL|_w^{1-c}.
    \end{equation*}
    We observe that $\CP_{k-1}$ is a $(p^{-(k-1)},\alpha,p^{\alpha}K_{\alpha})$-set of points in $R_{k-1}^2$, and $\CL_{k-1}$ is a $(p^{-(k-1)},\beta,p^{\beta}K_{\beta})$-set of lines in $R_{k-1}^2$. Notice that the total sum of weights is invariant under thickening:
    \begin{equation*}
        |\CP_{k-1}|_w = |\CP|_w \quad \text{and}\quad |\CL_{k-1}|_w = |\CL|_w.
    \end{equation*}
    Applying the induction hypothesis to $\CP_{k-1}$ and $\CL_{k-1}$ yields
    \begin{align*}
        I_w(\CP_{k-1},\CL_{k-1}) & \leq C(p,\varepsilon)p^{(k-1)(c+\varepsilon)} (K_{\alpha}K_{\beta}p^{\alpha+\beta})^c|\CP_{k-1}|_w^{1-c}|\CL_{k-1}|_w^{1-c}\\
        & \leq C(p,\varepsilon)p^{c(\alpha+\beta-1)}p^{(k-1)\varepsilon} p^{kc}(K_{\alpha}K_{\beta})^c |\CP|_w^{1-c}|\CL|_w^{1-c}.
    \end{align*}
    Thus, we have 
    \begin{align*}
        I_w(\CP,\CL) & \leq p^{kc}(K_{\alpha}K_{\beta})^{c} |\CP|_w^{1-c}|\CL|_w^{1-c} + p^{-1}\left(C(p,\varepsilon)p^{c(\alpha+\beta-1)}p^{(k-1)\varepsilon} p^{kc}(K_{\alpha}K_{\beta})^c |\CP|_w^{1-c}|\CL|_w^{1-c} \right)\\
        & \leq \left(1 + C(p,\varepsilon)p^{c(\alpha+\beta-1)-1}p^{(k-1)\varepsilon}\right)p^{kc}(K_{\alpha}K_{\beta})^c |\CP|_w^{1-c}|\CL|_w^{1-c}\\
        & \leq C(p,\varepsilon) p^{k(c+\varepsilon)}(K_{\alpha}K_{\beta})^c |\CP|_w^{1-c}|\CL|_w^{1-c},
    \end{align*}
    which completes the induction step and thus the proof.
\end{proof}

\smallskip

\section{Acknowledgements}
This work was supported by a Work Learn International Undergraduate Research Award at the University of British Columbia and by NSERC Discovery Grant 22R80520. The author is grateful to her supervisor Dr. Izabella \L aba for her guidance and support throughout the project, and for her invaluable feedback on multiple drafts of this paper. The author also thanks Paige Bright, Matthew Bull-Weizel, and Sushrut Tadwalkar for helpful conversations.

\smallskip

\bibliographystyle{amsplain}

\end{document}

%% file: preamble.tex
\usepackage[margin=1in]{geometry}                
\usepackage{amsmath}
\usepackage{amssymb}
\usepackage{amsthm}
\usepackage{amsfonts}
\usepackage{mathrsfs}
\usepackage{hyperref}
\hypersetup{
    colorlinks=true,
    linkcolor={blue!60!black},
    citecolor={blue!60!black},
    urlcolor={blue!90!black}
    filecolor=magenta,
    pdftitle={Overleaf Example},
    pdfpagemode=FullScreen,
    linktoc=all,
    }

\usepackage{url}
\usepackage{graphicx}
\usepackage{enumitem}

\usepackage{xcolor}
\usepackage{cleveref}
\usepackage{comment}
\usepackage{float}
\usepackage{enumitem}
\usepackage{mathtools}
\usepackage{arydshln}

\usepackage[utf8]{inputenc}
\usepackage[T1]{fontenc}

\numberwithin{equation}{section}

\allowdisplaybreaks
\newcommand{\ZZ}{\mathbb{Z}}
\newcommand{\NN}{\mathbb{N}}
\newcommand{\CC}{\mathbb{C}}
\newcommand{\RR}{\mathbb{R}}

\newcommand{\FF}{\mathbb{F}}
\newcommand{\PP}{\mathbb{P}}

\newcommand{\CA}{\mathcal{A}}

\newcommand{\CT}{\mathcal{T}}

\newcommand{\CL}{\mathcal{L}}

\newcommand{\CP}{\mathcal{P}}

\newcommand{\CI}{\mathcal{I}}
\newcommand{\CQ}{\mathcal{Q}}

\definecolor{blue-violet}{rgb}{0.54, 0.17, 0.89}

\newtheoremstyle{dotless}{}{}{\itshape}{}{\bfseries}{}{ }{}
\theoremstyle{dotless}

\newtheorem{theorem}{Theorem}[section]

\newtheorem{lemma}[theorem]{Lemma}
\newtheorem{proposition}[theorem]{Proposition}

\theoremstyle{remark}
\newtheorem*{remark}{Remark}

\theoremstyle{definition}
\newtheorem{definition}{Definition}[section]